\documentclass[11pt,oneside]{amsart}
\usepackage[a4paper, left=2.5cm, right=2.5cm, top=2.5cm, bottom=2.5cm]{geometry}
\usepackage[utf8]{inputenc}

\usepackage{amsmath}
\usepackage{amsfonts}
\usepackage{amssymb}

\usepackage{graphicx}
\usepackage{enumerate}
\usepackage{tikz-cd}

\newtheorem{theorem}{Theorem}[section]
\newtheorem{lemma}[theorem]{Lemma}
\newtheorem{corollary}[theorem]{Corollary}

\theoremstyle{definition}
\newtheorem{definition}[theorem]{Definition}
\newtheorem{example}[theorem]{Example}
\newtheorem{notation}[theorem]{Notation}

\theoremstyle{remark}
\newtheorem{remark}[theorem]{Remark}

\numberwithin{equation}{section}

\newcommand{\PP}{\mathbb{P}}

\newcommand{\NN}{\mathbb{N}}
\newcommand{\QQ}{\mathbb{Q}}

\newcommand{\Ii}{\mathcal{I}}
\newcommand{\Oo}{\mathcal{O}}

\newcommand{\adim}{\textnormal{adim}}
\newcommand{\vdim}{\textnormal{vdim}}
\newcommand{\Tr}{\textnormal{Tr}}
\newcommand{\Res}{\textnormal{Res}}
\newcommand{\HP}{\textnormal{HP}}
\newcommand{\HF}{\textnormal{HF}}

\newcommand{\bbinom}[2]{\genfrac{(}{)}{0pt}{1}{#1}{#2}}

\begin{document}

\author{Natalia Kupiec}
\address{Faculty of Mathematics and Computer Science, Jagiellonian University,
\L{}ojasiewicza 6, PL-30-348 Krak\'ow, Poland}
\email{natalia.kupiec@doctoral.uj.edu.pl}
\title[Actual and virtual dimension]{Actual and virtual dimension of codimension 2\\
general linear subspaces in $\mathbb{P}^n$}
%\date{\today}
\subjclass[2010]{14N20, 14C20, 14E07, 14N05}
\keywords{unexpected varieties, special linear systems, line arrangements, fat flats}%

\begin{abstract}
In the paper we compute the virtual dimension (defined by the Hilbert polynomial) of a space of hypersurfaces of given degree containing $s$ codimension 2 general linear subspaces in $\mathbb{P}^n$. We use Veneroni maps to find a family of unexpected hypersurfaces (in the style of B. Harbourne, J. Migliore, U. Nagel, Z. Teitler) and rigorously prove and extend examples presented in the paper by B. Harbourne, J. Migliore and H. Tutaj-Gasi\'nska.
\end{abstract}

\maketitle
\section{Introduction}
While examining hypersurfaces containing general linear subspaces in $\PP^n$ (possibly with multiplicities), the important question is how many such hypersurfaces exist. One can ask what is the dimension of the vector space of forms of given degree vanishing along those subspaces. We call it the \emph{actual dimension}. On the other hand, we can try to find the number of conditions that are imposed on such forms. This leads to the notion of the \emph{virtual dimension}, which is defined as the dimension of the vector spaces of all forms of degree $t$ minus the Hilbert polynomial of a scheme defined by vanishing along given subspaces. The interesting issue is whether those two values are equal.\par\medskip

The reason to focus on linear subspaces of codimension 2 is that these are being widely examined in low dimensions. In $\PP^2$ such subspaces are points. The virtual dimension of linear systems of forms of degree $t$ vanishing at $s$ general points with multiplicities $m_1,\ldots,m_s$ is equal
$$\binom{t+2}{2}-\sum_{i=1}^s\binom{m_i+2}{2}.$$
However, classification of systems for which the virtual and actual dimension are equal is the subject of SHGH (Segre-Harbourne-Gimigliano-Hirschowitz) Conjecture (see \cite{SHGH} for a survey on this topic).
\par\medskip

The next step is to consider lines in $\PP^3$. In this case it is still easy to find the virtual dimension of linear system of forms of degree $t$ vanishing along $s$ general lines with multiplicities $m_1,\ldots,m_s$. As presented in \cite{Linear} it is 
$$\binom{t+3}{3}-\frac{1}{6}\sum_{i=1}^sm_i(m_i+1)(3d+5-2m_i).$$
When all multiplicities are equal 1, the actual dimension is equal to the virtual dimension (as long as the virtual dimension is non-negative). This fact is known as Hartshorne-Hirschowitz Theorem \cite{HH}. For higher multiplicities there are examples when this is no longer true (cf. \cite{Unexp}).\par\medskip

When we study planes in $\PP^4$, it is more problematic to find the virtual dimension, since every two general planes intersect at a point. The intuitive way to take it into account in calculations is to use the inclusion-exclusion principle. Hence, for $s$ planes with multiplicities 1, the virtual dimension should be
$$\binom{t+4}{4}-s\binom{t+2}{2}+\binom{s}{2}.$$
It is also much harder to consider planes with multiplicities. For example, in \cite{Linear} authors avoid this problem by taking linear subspaces in $\PP^n$ with such codimensions that they do not intersect. Moreover, comparing the actual and virtual dimension led to the definition of unexpected hypersurfaces. As introduced in \cite{CHMN} and \cite{HMNT}, a scheme \emph{admits an unexpected hypersurface} of given degree, when the actual dimension is strictly greater than the virtual one.\par\medskip

One purpose of this paper is to prove the formula for the virtual dimension of the system of general linear subspaces of codimension 2 with multiplicities 1 in $\PP^n$. As in the case of $\PP^4$, it is given by inclusion-exclusion principle. Actually, we need to consider all the intersections of given subspaces to properly find the Hilbert polynomial. We present all the details in Section 3.\par\medskip

In Section 4 we introduce the Veneroni maps. They are generalizations of Cremona transformations in $\PP^2$ and cubo-cubic transformation in $\PP^3$. These birational maps of $\PP^n$ allow us to transform linear systems of forms of degree $n+k$ ($k\geq 3$) vanishing along $n+1$ codimension 2 linear subspaces with multiplicities 1 to forms of degree $kn+1$ vanishing with multiplicities $k$ along such subspaces. In \cite{Unexp} and \cite{New} there are presented examples of using those transformations that result in the examples of unexpected hypersurfaces. Another purpose of this paper is to generalize these results to higher dimensional projective spaces. We obtain a family of examples in Theorem \ref{unexp}. What is more, the analysis of mentioned systems for $k\geq 4$ lead us to the situation, when the virtual dimension is greater than the actual dimension (as in Example \ref{example21}). We say that a scheme \emph{misses an expected hypersurface} in this case. Such a situation could not appear in $\PP^2$ and $\PP^3$, because considered subspaces are not intersecting.\par\medskip

In Section 5 we gather all the results about the systems with multiplicities 1. In Theorem \ref{mult1aqv} we present conditions that guarantee that the actual dimension is greater than the virtual dimension and show in Example \ref{example4} that in $\PP^4$ without these conditions it may not be true. Later on, we compare dimensions in case of $n+1$ subspaces with a view towards using the Veneroni maps. It is summarized in Corollary \ref{mult1eq}. Section 6 focuses on the systems resulting from the transformation by the Veneroni maps. We prove Theorem \ref{mult3}, which is needed to show that there exists a family of unexpected hypersurfaces. In the last section we present the results regarding unexpected hypersurfaces including the examples of missing expected hypersurfaces.

\section{Basic definitions and auxiliary results}
Let $K$ be an algebraically closed field of characteristic 0 and $R=K[\PP^n]=K[x_0,\ldots,x_n]$ be the homogeneous coordinate ring of $\PP^n$. We consider $s$ distinct linear subspaces $\Lambda_1,\ldots,\Lambda_s\subset\PP^n$ of codimension 2 and denote by $X=m_1\Lambda_1+\ldots+m_s\Lambda_s$ the scheme defined by the ideal $$I_{X}=I(\Lambda_1)^{m_1}\cap\ldots\cap I(\Lambda_s)^{m_s}\subseteq R$$ generated by homogeneous polynomials vanishing on each $\Lambda_i$ to order at least $m_i$. We say that $m_i$ are the multiplicities of $X$.

\begin{definition}
The \emph{Hilbert function} of $X$ is the function $$\HF_X :\NN\ni t\mapsto \dim_K [R/I_X]_t\in\NN,$$ where subscript $[\ ]_t$ denotes degree $t$ part of graded homogeneous ring.
\end{definition}

It is well-known that there exists a polynomial $\HP_X\in\QQ[t]$ (called the \emph{Hilbert polynomial} of $X$) such that $\HP_X(t)=\HF_X(t)$ for all $t$ sufficiently large (see e.g. \cite{Eis}, \cite{Har}). To examine $\HP_X$ and $\HF_X$ for $X$ given by the vanishing along linear subspaces we need the following lemma that is proved in \cite{Linear}.
\begin{lemma}\label{cond}
Let $\Lambda\subset\PP^n$ be a linear subspace of codimension $k$. Let $t\geq m$ be a positive integer. Then vanishing to order at least $m$ along $\Lambda$ imposes 
$$c_{n,k,m,t}=\sum_{i=0}^{m-1}\binom{i+k-1}{k-1}\binom{t-i+n-k}{n-k}$$
linearly independent conditions on forms of degree $t$. Hence, for $m=1$ that gives $c_{n,k,1,t}=\bbinom{t+n-k}{n-k}$ conditions.
\end{lemma}
\noindent Obviously $\HF_{m\Lambda}(t)=c_{n,k,m,t}$ for $t\geq m$ and $\HP_{m\Lambda}(t)=c_{n,k,m,t}$ for any $t>0$.\par\medskip

We also introduce the following numbers associated to $X=m_1\Lambda_1+\ldots+m_s\Lambda_s\subseteq\PP^n$.
\begin{definition}\ \\
The \emph{actual dimension} of $X$ is the dimension of the vector space of forms in $I_X$ of degree $t$:
$$\adim_n(X,t)=\dim_K[I_X]_t=\dim_K[R]_t-\HF_X(t).$$
The \emph{virtual dimension} is defined as:
$$\vdim_n(X,t)=\dim_K[R]_t-\HP_X(t).$$
We say that $X$ \emph{admits an unexpected hypersurface} of degree $t$, if $ \adim_n(X,t)>0 $ and $$ \adim_n(X,t)>\vdim_n(X,t)$$ and $X$ \emph{misses an expected hypersurface} of degree $t$, when $$ \vdim_n(X,t)> \adim_n(X,t)>0 .$$
\end{definition}

\begin{remark}Usually along with the definition of  actual and virtual dimensions authors introduce the \emph{expected dimension}, which is the maximum of the virtual dimension and 0 (as in \cite{CHMN}, \cite{HMNT} and \cite{New}). Example \ref{example21} shows that the behaviour of the virtual dimension may be intricate and in many cases we cannot actually expect it to be equal to the actual dimension. Moreover, the situation presented in Example \ref{example21} allows us to introduce a new concept of missing expected hypersurfaces.
\end{remark}

We present some facts that are helpful in studying the actual and virtual dimesion. First, we need to introduce the residual and trace schemes.
\begin{definition}
Let $X,Y$ be closed subschemes of $\PP^n$ and $\Ii_X$, $\Ii_Y$ denote corresponding ideal sheafs.
The \emph{trace} of $X$ with respect to $Y$ is the scheme $\Tr_Y(X)$ defined by the ideal sheaf $(\Ii_X+\Ii_Y)/\Ii_Y$ (i.e. the schematic intersection of $X$ and $Y$ in $Y$).
The \emph{residual} of $X$ with respect to $Y$ is the scheme $\Res_Y(X)$ defined by the colon ideal $(\Ii_X:\Ii_Y)$.
\end{definition}

Let $X\subseteq\PP^n$ be a closed subscheme and $H\subseteq\PP^n$ be a hyperplane. We have the following exact sequence:
%$$0\to\Ii_{\Res_H(X)}(t-1)\to\Ii_X(t)\to\Ii_{\Tr_H(X)}(t)\to0.$$
$$0\to\Oo_{\PP^n}(t-1)\otimes\Ii_{\Res_H(X)}\to\Oo_{\PP^n}(t)\otimes\Ii_X\to\Oo_{H}(t)\otimes\Ii_{\Tr_H(X)}\to0.$$
We know that $\dim_K([I]_t)=h^0(\PP^n,\Oo_{\PP^n}(t)\otimes\Ii_X)$. So from the long exact sequence of cohomology we have the following inequality, which is called the Castelnuovo Inequality (more details can be found in \cite{AH} and \cite{AH2}).
\begin{lemma}[Castelnuovo Inequality]\label{castineq} For a closed subscheme $X\subseteq\PP^n$ and a hyperplane $H\subseteq\PP^n$ we have
$$\dim_K ([I_{X}]_t)\leq \dim_K ([I_{\Res_H(X)}]_{t-1})+\dim_K ([I_{\Tr_H(X),H}]_{t}),$$
where $I_{\Tr_H(X),H}$ denotes the ideal of forms in $H$ vanishing on $\Tr_H(X)$. Moreover, since $H\cong\PP^{n-1}$, we have 
$$\adim_n (X,t)\leq \adim_n (\Res_H(X),t-1)+\adim_{n-1} (\Tr_H(X),t).$$
\end{lemma}

If we take $t\gg0$, we have that $h^1(\PP^n,\Ii(t))=0$ by Serre Vanishing Theorem (cf. \cite{Hart}). Hence, we obtain the next result.
\begin{lemma}\label{casteq} For a closed subscheme $X\subseteq\PP^n$ and a hyperplane $H\subseteq\PP^n$ we have
$$\dim_K ([I_{X,\PP^n}]_t)= \dim_K ([I_{\Res_H(X),\PP^n}]_{t-1})+\dim_K ([I_{\Tr_H(X),H}]_{t})$$
for sufficiently large $t$. Moreover, since $H\cong\PP^{n-1}$, then 
$$\vdim_n (X,t)= \vdim_n (\Res_H(X),t-1)+\vdim_{n-1} (\Tr_H(X),t).$$
\end{lemma}

\section{The virtual dimension of schemes with multiplicities 1}
In this section we prove the formula for the virtual dimension of $X=\Pi_1+\ldots+\Pi_s$ which is a subscheme of $\PP^n$ ($n\geq 2, s\geq 0$), where $\Pi_1,\ldots, \Pi_s$ are general linear subspaces of codimension 2. By Lemma \ref{cond}, one such subspace gives $c_{n,2,1,t}=\bbinom{t+n-2}{n-2}$ conditions. However, every $i\leq s$ subspaces intersect in a subspace of codimension $2i$ as long as $2i \leq n$. There are $\binom{s}{i}$ such intersections each giving $\bbinom{t+n-2i}{n-2i}$ conditions. By considering all possible intersections we find the formula for the expected number of conditions that vanishing on $X$ imposes on forms of degree $t$. We denote it by $$S_{n,s,t}=\sum_{i=0}^{N(n,s)}(-1)^i\binom{s}{i}\binom{t+n-2i}{n-2i},$$ where for given $n, s$ we take $N(n,s)=\min\{\lfloor n/2\rfloor, s\}$. Now, we present an important property of introduced formula that is proved in the Appendix.

\begin{lemma}\label{combi1}
The following formula holds $$S_{n,s,t}=S_{n,s-1,t-1}+S_{n-1,s-1,t-1}.$$
\end{lemma}

We move on to the main theorem of this section.
\begin{theorem}\label{thm1} Let $\Pi_1,\ldots,\Pi_s\subseteq\PP^n$ be distinct general linear subspaces of codimension 2. Then for $X=\Pi_1+\ldots+\Pi_s$ we have $\vdim_n(X,t)=S_{n,s,t}.$ 
\end{theorem}

\begin{proof} Notice that both $\vdim_n(X,t)$ and $S_{n,s,t}$ are polynomials with respect to $t$.\\
We will proceed by induction on the sum $n+s$.
We assume that $n+s=2$, so $n=2$ and $s=0$. Then $$\vdim_2(X,t)=\adim_2(X,t)=\bbinom{t+2}{2}=S_{2,0,t}.$$ If $n+s=3$, there are two possibilities. For $n=2$ and $s=1$, we have $$\vdim_2(X,t)=\adim_2(X,t)=\bbinom{t+2}{2}-1=S_{2,1,t}$$ and for $n=3$ and $s=0$, we get $$\vdim_3(X,t)=\adim_3(X,t)=\bbinom{t+3}{3}=S_{3,0,t}.$$
\indent Now we assume that $n+s=k$ and that polynomials $\vdim_n(X,t)$ and $S_{n,s,t}$ are equal for $n+s=k-1$ and $n+s=k-2$.\\
Choose $H\cong\PP^{n-1}\subseteq\PP^n$ in such a way that $\Pi_s\subseteq H$. Then $$\Res_H(X)=\Pi_1+\ldots+\Pi_{s-1}$$ and $$\Tr_H(X)=\Pi_1|_H+\ldots+\Pi_{s}|_H.$$ By Lemma \ref{casteq} we have
$$\vdim_n(X,t)=\vdim_n(\Pi_1+\ldots+\Pi_{s-1},t-1)+\vdim_{n-1}(\Pi_1|_H+\ldots+\Pi_{s}|_H,t).$$
Since $\Pi_s\subseteq H$ is a fixed component for the locus of forms in $H$, we have $$\vdim_{n-1}(\Pi_1|_H+\ldots+\Pi_{s}|_H,t)=\vdim_{n-1}(\Pi_1|_H+\ldots+\Pi_{s-1}|_H,t-1).$$ By induction and Lemma \ref{combi1} we obtain our result:
\begin{gather*}
\vdim_n(X,t)=\vdim_n(\Pi_1+\ldots+\Pi_{s-1},t-1)
+\vdim_{n-1}(\Pi_1|_H+\ldots+\Pi_{s-1}|_H,t-1)\\
=S_{n,s-1,t-1}+S_{n-1,s-1,t-1}=S_{n,s,t}.
\end{gather*}
\end{proof}

\section{Veneroni maps} Now we recall the construction of Veneroni maps. One can find more details in \cite{Ven}, where authors present basic facts about Veneroni maps with a modern approach. In \cite{Unexp} there is a thorough analysis of the case of a cubo-cubic transformation in $\PP^3$, which is a special case of the Veneroni map.\par\medskip

Let $\Pi_1,\ldots,\Pi_{n+1}\subseteq\PP^n$ be codimension 2 general linear subspaces. The linear system of all forms of degree $n$ vanishing on those subspaces has dimension $n+1$. Hence, it defines a rational map $v_n:\PP^n\dashrightarrow\PP^n$. This map is called the \emph{Veneroni map}. As it is proved in \cite{Ven}, its inverse is again the Veneroni map defined by the degree $n$ forms vanishing on some general linear subspaces $\Pi'_1,\ldots,\Pi'_{n+1}$ of codimension 2 and $v_n$ is in fact a birational map.\par\medskip

Following the notation from \cite{New} we write $dH-m_1\Pi_1-\ldots-m_{n+1}\Pi_{n+1}$ to denote the linear system of forms of degree $d$ that vanish on each $\Pi_i$ to order at least $m_i$. Let $H$ and $H'$ denote the linear system of all hyperplanes in $\PP^n$, where $H$ is considered in the source of the map $v_n$ and $H'$ in its target. Then $v_n$ pulls back $H'$ to $nH-\Pi_1-\ldots-\Pi_{n+1}$. By the construction of $\Pi_i'$ from \cite{Ven}, $v_n$ pulls back $H'$ to a unique hypersurface of degree $n-1$ containing $\Pi_1,\ldots,\Pi_{j-1},\Pi_{j+1},\ldots,\Pi_{n+1}$, which is denoted by $(n-1)H-\Pi_1-\ldots-\Pi_{n+1}+\Pi_j$.\par\medskip

To generalize the results from \cite{New} we examine the linear systems of the form $S=(n+k)H-\Pi_1-\ldots-\Pi_{n+1}$. Then 
\begin{multline*}
v_n^{*}(S)=S'=(n+k)H'-\Pi'_1-\ldots-\Pi'_{n+1}\\
=(n+k)(nH-\Pi_1-\ldots-\Pi_{n+1})-((n-1)H-\Pi_2-\ldots-\Pi_{n+1})\\
-\ldots -((n-1)H-\Pi_1-\ldots \Pi_n)
=(nk+1)H-k\Pi_1-\ldots-k\Pi_{n+1}.
\end{multline*}

\noindent Since $v_n$ is a birational map, we have $$\adim_n(\Pi_1+\ldots+\Pi_{n+1},n+k)=\adim_n(k\Pi_1+\ldots+k\Pi_{n+1},kn+1).$$
In the following sections we focus on those two types of linear systems. First, we will examine the actual dimension of $S$ (and therefore of $S'$). Then, we will analyse the virtual dimension of $S'$ for $k=3$. As a consequence, we will show that the systems $S'$ for $k=3$ form a family of examples of unexpected hypersurfaces.

\section{Linear systems with base loci in subspaces of codimension 2}
This section presents the analysis of the relation between the virtual and actual dimension of the linear systems of hypersurfaces containing general linear subspaces of codimension 2 with multiplicities 1. We prove Theorem \ref{mult1aqv} that provides sufficient conditions under which the actual dimension is greater than the virtual dimension. Afterwards, we show that that systems of hypersurfaces of degree $n+k$ vanishing along $n+1$ such subspaces in $\PP^n$ satisfy those conditions and that the virtual dimension is positive in such case. Then, we check that the opposite equality between dimensions also holds in this situation.

\begin{notation}It is convenient to introduce, for numbers $a,b,c$ the following notation:
\begin{small}
\begin{center}
\begin{tikzcd}
a \arrow[r,dashed] \arrow[d,dashed] & b\\
c
\end{tikzcd}
\end{center}
\end{small}
to denote that $a\leq b+c$ and 
\begin{small}
\begin{center}
\begin{tikzcd}
a \arrow[r] \arrow[d] & b\\
c
\end{tikzcd}\\
\end{center}
\end{small}
for $a=b+c$.
\end{notation}

The following example was presented in \cite[Ex. 4.2]{New}. Hereby, we show how to calculate the actual dimension of a given system without using computer methods.

\begin{example}\label{mult1example} 
Let $X=\Pi_1+\ldots+\Pi_5\subseteq\PP^4$. We want to show that $$\adim_4(X,7)=\vdim_4(X,7)=160.$$ By Theorem \ref{thm1} we have $\vdim_4(X,7)=S_{4,5,7}=\bbinom{7+4}{4}-5\bbinom{7+2}{2}+\bbinom{5}{2}=160$.\par
First, we prove that $\adim_4(X,7)\geq 160$. It is enough to show that we can properly construct the sets $W_i$ of linear equations given by the vanishing of forms of degree 7 along each $\Pi_i$. First, we construct the sets $W_{ij}$ of linearly independent equations given by the vanishing along intersections $\Pi_i\cap\Pi_j$. There are $\bbinom{5}{2}=10$ such intersections and those are points in this case, hence each gives one equation. Now each $W_i$ should contain the equations that are already in sets $W_{ij}$ for $j\neq i$, hence 4 equations, whereas one codimension 2 subspace in $\PP^4$ gives $\bbinom{7+2}{2}=36$ equations. Then, we can add $36-4$ linearly independent equations to each set $W_i$. In conclusion, by inclusion-exclusion principle
$$\adim_4(X,7)\geq \bbinom{7+4}{4} -|W_1\cup\ldots\cup W_5|=\bbinom{7+4}{4}- \sum_{i=1}^{5}|W_i|+\sum_{1\leq i\leq j\leq 5}|W_{ij}|=160.$$
\indent On the other hand, we take a hyperplane $H_5\cong\PP^3$ in such a way that $\Pi_5\subseteq H_5$.
Then
$$\Res_{H_5}(X)=\Pi_1+\ldots+\Pi_4$$ and $$\Tr_{H_5}(X)=\Pi_1|_{H_5}+\ldots+\Pi_{5}|_{H_5}.$$ By Lemma \ref{castineq} we get
$$\adim_4(X,t)\leq\adim_4(\Res_{H_5}(X),t-1)+\adim_{3}(\Tr_{H_5}(X),t).$$
And as before $$\adim_{3}(\Tr_{H_5}(X),t)=\adim_{3}(\Tr_{H_5}(X)-\Pi_5|_{H_5},t-1)$$ because $\Pi_5$ is a component for the locus of forms in $H_5$. Moreover, by Hartshorne-Hirschowitz Theorem we have $$\adim_{3}(\Tr_{H_5}(X)-\Pi_5|_{H_5},6)=\vdim_3(\Tr_{H_5}(X)-\Pi_5|_{H_5},6)=S_{3,4,6}=56.$$ We illustrate the situation as a diagram.

\begin{small}
\begin{center}
\begin{tikzcd}
\adim_4(X,7) \arrow[r,dashed] \arrow[d,dashed] & \adim_{3}(\Pi_1|_{H_5}+\ldots+\Pi_{4}|_{H_5},6)=S_{3,4,6}=56\\
\adim_4(\Pi_1+\ldots+\Pi_4,6)
\end{tikzcd}\\
\end{center}
\end{small}

\noindent We continue that procedure by taking $X'=\Pi_1+\ldots+\Pi_4$ and $H_4$ having $\Pi_4$ as a subset.

\begin{small}
\begin{center}
\begin{tikzcd}
\adim_4(\Pi_1+\ldots+\Pi_4,6) \arrow[r,dashed] \arrow[d,dashed] & \adim_{3}(\Pi_1|_{H_4}+\ldots+\Pi_{3}|_{H_4})=S_{3,4,6}=38\\
\adim_4(\Pi_1+\ldots+\Pi_3,5)
\end{tikzcd}
\end{center}
\end{small}

\noindent Then we take consecutively $H_3\supseteq\Pi_3$ and $H_2\supseteq\Pi_2$.

\begin{small}
\begin{center}
\begin{tikzcd}
\adim_4(\Pi_1+\ldots+\Pi_3,5) \arrow[r,dashed] \arrow[d,dashed] & \adim_{3}(\Pi_1|_{H_3}+\Pi_{2}|_{H_3},4)=25\\
\adim_4(\Pi_1+\Pi_2,4)\arrow[r,dashed] \arrow[d,dashed] & \adim_{3}(\Pi_1|_{H_2},4)=16\\
\adim_4(\Pi_1,3)
\end{tikzcd}
\end{center}
\end{small}

\noindent That gives us the inequality:
$$\adim_4(X,7)\leq 135+\adim_4(\Pi_1,3).$$
Since $\adim_4(\Pi_1,3)=\bbinom{7}{4}-\bbinom{5}{2}=25$, we have $\adim_4(X,7)\leq 160$.
\end{example}

\begin{theorem}\label{mult1aqv}
Let $\Pi_1,\ldots,\Pi_s$ be general linear subspaces of codimension 2 in $\PP^n$. Consider the scheme $X=\Pi_1+\ldots+\Pi_s$. If $$S_{n-2p,s-p,t}>0$$
for $p=1,\ldots, N(n,s)-1$, then $$\adim_n(X,t)\geq\vdim_n(X,t).$$
\end{theorem}

\begin{proof}
Recall that 
\begin{gather*}
S_{n-2p,s-p,t}= \sum_{i=0}^{N(n,s)-p}(-1)^{i}\binom{s-p}{i}\binom{t+n-2p-2i}{n-2p-2i}\\
=\sum_{i=p}^{N(n,s)}(-1)^{i-p}\binom{s-p}{i-p}\binom{t+n-2i}{n-2i},
\end{gather*}
where $N(n,s)=\min\{\lfloor n/2\rfloor, s\}$. Also notice that vanishing along the subspace of codimension $c$ imposes $\bbinom{t+n-c}{n-c}$ conditions on forms of degree $t$.\par

We proceed as in Example \ref{mult1example}. We want to construct the sets $W_{i_1\ldots i_l}$ of linear equations given by vanishing of forms of degree $t$ along the intersections $\Pi_{i_1}\cap\ldots\cap\Pi_{i_l}$. Take $l=N(n,s)$, then $|W_{i_1\ldots i_l}|=\bbinom{t+n-2l}{n-2l}$ and any intersection of $l+1$ subspaces is empty. First, we set up such sets of linearly independent conditions. Then for $p=l-1$ each intersection of $p$ subspaces contains $s-p$ intersections of $p+1=l$ subspaces. We want to have $$\bbinom{t+n-2p}{n-2p}- (s-p)\bbinom{t+n-2l}{n-2l}=S_{n-2p,s-p,t}>0,$$ since in such case we can add $S_{n-2p,s-p,t}$ equations to $W_{i_1\ldots i_p}$ and have it defined properly.\par

Similarly, for any $p=1,\ldots, N(n,s)-2$ we use the equations that were previously defined for smaller subspaces and complete them to the set $W_{i_1\ldots i_p}$.  We use the inclusion-exclusion principle and that results in the inequalities of the form 
\begin{gather*}
\bbinom{t+n-2p}{n-2p}-(s-p)|W_{i_1\ldots i_{p-1}}|+\bbinom{s-p}{2}|W_{i_1\ldots i_{p-2}}|-\ldots+(-1)^{l-p}\bbinom{s-p}{l-p}|W_{i_1\ldots i_l}|\\
=S_{n-2p,s-p,t}>0.
\end{gather*}
Hence, if inequalities $S_{n-2p,s-p,t}>0$ are satisfied for $p=1,\ldots, N(n,s)-1$, then by inclusion-exclusion principle we have that
$$\adim_n(X,t)\geq \bbinom{t+n}{n}-|W_1\cup\ldots\cup W_s|=S_{n,s,t}=\vdim_n(X,t).$$
\end{proof}

In the following example we show that the conditions of the form $S_{n-2p,s-p,t}>0$ are needed in the previous theorem.

\begin{example}\label{example4} Consider $X=\Pi_1+\ldots+\Pi_{12}$ in $\PP^4$. There are no forms of degree 2 vanishing on 12 general planes. Indeed, taking a hyperplane $H\cong \PP^3$ such that $\Pi_{12}\subseteq H$ and using Lemma \ref{castineq} we get
$$\adim_4(X,2)\leq\adim_4(\Pi_1+\ldots+\Pi_{11},1)+\adim_3(\Pi_1|_H+\ldots+\Pi_{11}|_H,1)=0.$$
On the other hand, by Theorem \ref{thm1} 
$$\vdim_4(X,2)=\bbinom{2+4}{4}-12\bbinom{2+2}{2}+\bbinom{12}{2}=9.$$
The assumptions of Theorem \ref{mult1alv} are not satisfied, since $S_{2,11,2}=-5.$
\end{example}

For the rest of this section we assume that $X=\Pi_1+\ldots+\Pi_{n+1}\subseteq\PP^n$, where $\Pi_1,\ldots,\Pi_{n+1}$ are general linear subspaces of codimension 2 and $n\geq 2$. We want to compare the virtual and actual dimension of the system of forms of degree $n+k$ vanishing along $X$. For this purpose, we prove the next two theorems. %First, we want to use the previous theorem in such case. Hence, we prove the following.

\begin{theorem}\label{mult1cond}
If $k\geq 3$, then we have $$S_{n-2p,n+1-p,n+k}>0$$
for $p=0,\ldots, \lfloor n/2\rfloor-1$.
\end{theorem}

\begin{proof}
We will prove by induction on $n$ that $$S_{n-2p,n+1-p-j,n+k-j}>0$$
for $j=0,\ldots,n-p$.\\
For $n=2$ we need to consider $p=0$ and $j=0,1,2$ and we get
$$S_{2,3,k+2}=\bbinom{k+4}{2}-3>0,\quad S_{2,2,k+1}=\bbinom{k+3}{2}-2>0,\quad S_{2,1,k}=\bbinom{k+2}{2}-1>0.$$
Fix $n$ and $p$. First, we consider the case $j=n-p$ and show that $S_{n-2p,1,k+p}>0$. Indeed, notice that for $\tilde{n}=n-2p\geq 2$ and $\tilde{k}=k+p\geq3$, we have 
$$S_{n-2p,1,k+p}=S_{\tilde{n},1,\tilde{k}}=\binom{\tilde{k}+\tilde{n}}{\tilde{n}}-\binom{\tilde{k}+\tilde{n}-2}{\tilde{n}-2}=\binom{\tilde{k}+\tilde{n}-2}{\tilde{n}-2}\cdot\left(\frac{(\tilde{k}+\tilde{n}-1)(\tilde{k}+\tilde{n})}{(\tilde{n}-1)\tilde{n}}-1\right)>0.$$
For $j\leq n-p-1$ by Lemma \ref{combi1}
$$S_{n-2p,n+1-p-j,n+k-j}= S_{n-2p,n+1-p-(j+1),n+k-(j+1)} +S_{(n-1)-2p,(n-1)+1-p-j,(n-1)+k-j}.$$
By reverse induction on $j$ we have $$S_{n-2p,n+1-p-(j+1),n+k-(j+1)}>0.$$
If $p\leq \lfloor \frac{n-1}{2}\rfloor-1$, we have $$S_{(n-1)-2p,(n-1)+1-p-j,(n-1)+k-j}>0$$ by induction on $n$. For odd $n$ we consider 
$p=0,\ldots,\lfloor \frac{n}{2}\rfloor-1=\lfloor \frac{n-1}{2}\rfloor-1$ and for even $n$ we take
$p=0,\ldots,\lfloor \frac{n}{2}\rfloor-1=\lfloor \frac{n-1}{2}\rfloor.$ However, in the case of $p=\lfloor\frac{n-1}{2}\rfloor$ for even $n$, we have 
$$S_{(n-1)-2p,(n-1)+1-p-j,(n-1)+k-j}=S_{0,(n-1)+1-p-j,(n-1)+k-j}=1.$$

\end{proof}

\begin{theorem}\label{mult1alv}
For $k\geq 3$, we have that $$\adim_n(X,n+k)\leq\vdim_n(X,n+k).$$
\end{theorem}

\begin{proof} First, notice that $$\adim_n(\Pi,k)=\vdim_n(\Pi,k)=S_{n,1,k},$$ where $\Pi$ is a codimension 2 linear subspace in $\PP^n$.
We will prove by induction on $n$ that
$$\adim_n(\Pi_1+\ldots+\Pi_{n+1-j},n+k-j)\leq S_{n,n+1-j,n+k-j}
$$
for $j=0,\ldots, n-1$. We proceed similarly to the proof of the previous theorem.\\
For $n=2$ the linear subspaces of codimension 2 are points, so
\begin{gather*}
\adim_2(\Pi_1+\Pi_2+\Pi_3,k+2)=\vdim_2(\Pi_1+\Pi_2+\Pi_3,k+2)=S_{2,3,k+2},\\
\adim_2(\Pi_1+\Pi_2,k+1)=\vdim_2(\Pi_1+\Pi_2,k+1)=S_{2,2,k+1}.
\end{gather*}
We fix $n$ and use Lemma \ref{castineq} by taking such hyperplane $H$ that $\Pi_{n+1-j}\subseteq H\cong\PP^{n-1}$. We get
\begin{multline*}
\adim_n(\Pi_1+\ldots+\Pi_{n+1-j},n+k-j)\leq \adim_n(\Pi_1+\ldots+\Pi_{n+1-(j+1)},n+k-(j+1))\\
+\ \adim_{n-1}(\Pi_1+\ldots+\Pi_{(n-1)+1-j},(n-1)+k-j).
\end{multline*}
By reverse induction on $j$ and induction on $n$ the right-hand side is not greater than 
$$S_{n,n-j,n+k-j+1}+S_{n-1,n-j,n+k-j+1}.$$
By Lemma \ref{combi1} it is equal to $S_{n,n+1-j, n+k-j}$.
\end{proof}

\begin{corollary}\label{mult1eq}
By Theorems \ref{mult1aqv}, \ref{mult1cond} and \ref{mult1alv} we have
$$\adim_n(X,n+k)=\vdim_n(X,n+k)>0.$$
for $k\geq 3$.
\end{corollary}

\section{Images of linear systems under Veneroni transformations}
Now we focus on the linear systems that are transformations of $$S=(n+k)H-\Pi_1-\ldots-\Pi_{n+1}$$ by the Veneroni maps, where $\Pi_1,\ldots,\Pi_{n+1}\subseteq\PP^n$ are codimension 2 general linear subspaces and $k\geq 3$. As we explained in Section 4, such systems are of the form $$S'=(nk+1)H-k\Pi_1-\ldots-k\Pi_{n+1}.$$
We will prove that for $k=3$ and $n\geq 6$ the virtual dimension of $S'$ is always non-positive.\par\medskip

First, we  continue the analysis of the example presented in \cite[Ex. 4.2.]{New}, which shows in particular that the assumption $n\geq 6$ is essential.
\begin{example}\label{mult3example}
Let $X'=3\Pi_1+\ldots+3\Pi_5\subseteq\PP^4$. We want to find $\vdim_4(X',13)$. As before, we assume that $H\cong\PP^3$ is such a hyperplane that $\Pi_5\subseteq H$. Then
$$\Res_{H}(X')=3\Pi_1+\ldots+3\Pi_4+2\Pi_5$$ and $$\Tr_{H}(X')=3\Pi_1|_{H}+\ldots+3\Pi_{5}|_{H}.$$
By Lemma \ref{casteq} and the fact that $3\Pi_{5}$ is a component for the zero locus of forms in $H$ we have
$$\vdim_4(X',t)=\vdim_4(\Res_{H}(X'),t-1)+\vdim_{3}(\Tr_{H}(X')-3\Pi_{5}|_H,t-3).$$
We repeat that procedure as presented in the diagram. To alleviate notation, at each level we write $\Pi_i|_H$, although at each step we take different hyperplanes $H_i$.

\begin{small}
\begin{center}
\begin{tikzcd}
\vdim_4(3\Pi_1+\ldots+3\Pi_5,13) \arrow[r] \arrow[d] & \vdim_{3}(3\Pi_1|_{H}+\ldots+3\Pi_{4}|_{H},10)\\
\vdim_4(3\Pi_1+\ldots+3\Pi_4+2\Pi_5,12)\arrow[r] \arrow[d] & \vdim_{3}(3\Pi_1|_{H}+\ldots+3\Pi_{4}|_{H},10)\\
\vdim_4(3\Pi_1+\ldots+3\Pi_4+\Pi_5,11)\arrow[r] \arrow[d] & \vdim_{3}(3\Pi_1|_{H}+\ldots+3\Pi_{4}|_{H},10)\\
\vdim_4(3\Pi_1+\ldots+3\Pi_4,10)\arrow[r] \arrow[d] & \vdim_{3}(3\Pi_1|_{H}+\ldots+3\Pi_{3}|_{H},7)\\
\vdim_4(3\Pi_1+3\Pi_2+3\Pi_3+2\Pi_4,9)\arrow[r] \arrow[d] & \vdim_{3}(3\Pi_1|_{H}+\ldots+3\Pi_{3}|_{H},7)\\
\vdim_4(3\Pi_1+3\Pi_2+3\Pi_3+\Pi_4,8)\arrow[r] \arrow[d] & \vdim_{3}(3\Pi_1|_{H}+\ldots+3\Pi_{3}|_{H},7)\\
\vdim_4(3\Pi_1+3\Pi_2+3\Pi_3,7)\arrow[r] \arrow[d] & \vdim_{3}(3\Pi_1|_{H}+3\Pi_{2}|_{H},4)\\
\vdots \arrow[d]\\
\vdim_4(3\Pi_1+3\Pi_2,4)\arrow[r] \arrow[d] & \vdim_{3}(3\Pi_1|_{H},1)\\
\vdots \arrow[d]\\
\vdim_4(3\Pi_1,1)
\end{tikzcd}
\end{center}
\end{small}

\noindent It follows that 
\begin{multline*}
\vdim_4(3\Pi_1+\ldots+3\Pi_5,13)=3\big[\vdim_{3}(3\Pi_1|_{H}+\ldots+3\Pi_{4}|_{H},10)\big]\\
+\ \vdim_4(3\Pi_1+\ldots+3\Pi_4,10)
\end{multline*}
and so forth. That gives us 
\begin{multline*}
\vdim_4(3\Pi_1+\ldots+3\Pi_5,13)=3\big[\vdim_{3}(3\Pi_1|_{H}+\ldots+3\Pi_{4}|_{H},10)\\
+\vdim_{3}(3\Pi_1|_{H}+\ldots+3\Pi_{3}|_{H},7)
+\vdim_{3}(3\Pi_1|_{H}+3\Pi_{2}|_{H},4)+\vdim_{3}(3\Pi_1|_{H},1)\big]\\+\vdim_4(3\Pi_1,1).
\end{multline*}
One can check that $\vdim_{3}(3\Pi_1|_{H},1)=\vdim_4(3\Pi_1,1)=0.$ To find remaining values we repeat the procedure in $\PP^3$, we denote $S_i:=\Pi_i|_{H}$ and hyperplanes in $\PP^3$ as $H'$.

%\newpage
\begin{small}
\begin{center}
\begin{tikzcd}
\vdim_{3}(3S_1+\ldots+3S_{4},10) \arrow[r] \arrow[d] & \vdim_{2}(3S_1|_{H'}+\ldots+3S_3|_{H'},7)=18\\
\vdots \arrow[d]\\
\vdim_{3}(3S_1+\ldots+3S_{3},7) \arrow[r] \arrow[d] & \vdim_{2}(3S_1|_{H'}+3S_2|_{H'},4)=3\\
\vdots \arrow[d]\\
\vdim_{3}(3S_1+3S_{2},4) \arrow[r] \arrow[d] & \vdim_{2}(3S_1|_{H'},1)=-3\\
\vdots \arrow[d]\\
\vdim_3(3S_1,1)=0
\end{tikzcd}
\end{center}
\end{small}
The numbers on the right are easily calculated, since we consider vanishing at points in $\PP^2$.
Hence, $\vdim_{3}(3S_1+3S_{2},4)=-9$, $\vdim_{3}(3S_1+\ldots+3S_{3},7)=0$ and $\vdim_{3}(3S_1+\ldots+3S_{4},10)=54$. As a consequence $\vdim_4(X',13)=135$.
\end{example}

\begin{remark}\label{rem:proc}In the case of multiplicities 1 we proved that there is a formula for the virtual dimension of a given system. When we consider higher multiplicities the situation becomes much more complicated. We can use Lemma \ref{casteq} to find the virtual dimension as in the previous example. The procedure that returns the value
$$\vdim_n(m_1\Pi_1+\ldots+m_s\Pi_{s},t)$$
for given $n\geq2$, $m_i>0$ and $s\geq0$ can be easily implemented in Singular or Maple.\\
Notice that if $\Pi_s\subseteq H\cong \PP^{n-1}$, then
$$\Res_{H}(m_1\Pi_1+\ldots+m_s\Pi_{s})=m_1\Pi_1+\ldots+(m_s-1)\Pi_{s}$$ and $$\Tr_{H}(m_1\Pi_1+\ldots+m_s\Pi_{s})-m_s\Pi_s=m_1\Pi_1+\ldots+m_{s-1}\Pi_{s-1}.$$
By Lemma \ref{casteq} $\vdim_n(m_1\Pi_1+\ldots+m_s\Pi_{s},t)$ is equal to
$$\vdim_n(m_1\Pi_1+\ldots+(m_s-1)\Pi_{s},t-1)+\vdim_{n-1}(m_1\Pi_1+\ldots+m_{s-1}\Pi_{s-1},t-m_s).$$
We define the procedure recursively starting from the conditions
$$\vdim_n(X,t)=\binom{t+ n}{n}$$
for $X=\varnothing$ and
$$\vdim_2(m_1\Pi_1+\ldots+m_s\Pi_{s},t)=\binom{t+ 2}{2}-\sum_{i=1}^s \binom{m_i+1}{2}.$$
\end{remark}

\begin{lemma}\label{multk}For $n\geq k\geq 3$ we have $$\vdim_n(k\Pi,1)=0.$$
\end{lemma}
\begin{proof} By Lemma \ref{cond} we have that $$\adim_n(k\Pi,t)=\binom{t+n}{n}-c_{n,2,k,t}=\binom{t+n}{n}-\sum_{i=0}^{k-1}(i+1)\binom{t-i+n-2}{n-2}$$ for $t\geq k$. Denote the summands in $c_{n,2,k,t}$ as $c_i(t)=(i+1)\binom{t-i+n-2}{n-2}$ treated as polynomials with respect to $t$. Then $$\vdim_n(k\Pi,1)=n+1-\sum_{i=0}^{k-1}c_i(1).$$\\
If $n>2$, then $c_0(1)=\binom{1+n-2}{n-2}=n-1$, $c_1(1)=2\cdot\binom{n-2}{n-2}=2$  and $c_i(1)=0$ for $i=2,\ldots k-1$. Indeed, $$c_i(t)=(i+1)\cdot \frac{(t-i+1)(t-i+2)\cdot\ldots\cdot (t-i+n-2)}{(n-2)!}$$
and $1\leq i-1\leq k-1\leq n-2$, so  one of the terms in the numerator is equal 0.
Hence, $\vdim_n(k\Pi,1)=n+1-(n-1)-2=0.$
\end{proof}

\begin{theorem} \label{mult3}For $n\geq 6$ we have that $$\vdim_n(3\Pi_1+\ldots+3\Pi_{n+1},3n+1)\leq0.$$
\end{theorem}

\begin{proof}First, we see that $\vdim_n(3\Pi,1)=0$ for $n\geq 3$ by Lemma \ref{multk}.
We will prove by induction on $n$ that 
$$\vdim_{n}(3\Pi_1+\ldots+3\Pi_{n+1-j},3(n-j)+1)<0$$
for $j=0,\ldots,n-1$.\\
Explicit calculation along the lines of Example \ref{mult3example} or appealing to the procedure explained in the previous remark that
\begin{gather*}
\vdim_6(3\Pi_1+\ldots+3\Pi_7,19)=0,\quad \vdim_6(3\Pi_1+\ldots+3\Pi_6,16)=-729,\\
\vdim_6(3\Pi_1+\ldots,3\Pi_5,13)=-243,\quad\vdim_6(3\Pi_1+\ldots+3\Pi_4,10)=0,\\
 \vdim_6(3\Pi_1+3\Pi_2+3\Pi_3,7)=0,\quad \vdim_6(3\Pi_1+3\Pi_2,4)=0.
\end{gather*}
Now we fix $n$ and use Lemma \ref{casteq} three times to obtain 
\begin{multline*}
\vdim_{n}(3\Pi_1+\ldots+3\Pi_{n+1-j},3(n-j)+1)
=\vdim_n(3\Pi_1+\ldots+3\Pi_{n+1-(j+1)},3(n-(j+1))+1)\\
+3\big(\vdim_{n-1}(3\Pi_1+\ldots+3\Pi_{(n-1)+1-j},3((n-1)-j)+1\big).
\end{multline*}
The right-hand side is smaller than 0 by reverse induction on $j$ and induction on $n$.\qed
\end{proof}

\section{Application to unexpected hypersurfaces}

In this section we present our results concerning unexpected hypersurfaces. Again we assume that $\Pi_1,\ldots,\Pi_{n+1}$ are codimension 2 general linear subspaces in $\PP^n$ and $X=\Pi_1+\ldots+\Pi_{n+1}$. We consider the linear system of forms of degree $n+k$ vanishing along $X$ and its transformation by the Veneroni map to the system of forms of degree $kn+1$ vanishing along $X'=k\Pi_1+\ldots+k\Pi_{n+1}$. First, we examine the case $k=3$.

\begin{theorem}\label{unexp}
Let $X'=3\Pi_1+\ldots+3\Pi_{n+1}$, where $n\geq 3$. Then $X'$ admits an unexpected hypersurface of degree $3n+1$.
\end{theorem}

\begin{proof} The linear system $S'=(3n+1)H-3\Pi_1-\ldots-3\Pi_{n+1}$ is the pullback of the system $S=(n+3)H-\Pi_1-\ldots-\Pi_{n+1}$ by the Veneroni map $v_n$. So $\adim_n(X',3n+1)=\adim_n(X,n+3)$ and by Corollary \ref{mult1eq} this value is strictly greater than 0.\par
For $n=3,4,5$ one can check that $$\adim_n(X',3n+1)=\adim_n(X,n+3)>\vdim_n(X',3n+1)$$ by straightforward calculation as presented in Example \ref{mult1example} and Example \ref{mult3example} or using implemented procedure described in Remark \ref{rem:proc}. For $n\geq 6$ by Theorem \ref{mult3} $\vdim_n(X',3n+1)\leq 0$. Hence, $X'$ admits an unexpected surface of degree $3n+1$ for all $n\geq 3$.
\end{proof}

When considering higher values of $k$, it turns out that not only unexpected hypersurfaces appear. Surprisingly, for $k=4$ the virtual dimension of a given scheme in $\PP^{21}$ exceeds the actual dimension.

\begin{example}\label{example21} Consider $Y=\Pi_1+\ldots+\Pi_{22}$ and $Y'=4\Pi_1+\ldots+4\Pi_{22}$ in $\PP^{21}$. The system of forms of degree 25 on $Y$ is transformed to the system of forms of degree 85 on $Y'$ by the Veneroni map. The results obtained using procedure from Remark \ref{rem:proc} are as follows. By Corollary \ref{mult1eq} we have
$$\adim_{21}(Y',85)=\adim_{21}(Y,25)=\vdim_{21}(Y,25)=1\,337\,982\,976,$$
whereas $$\vdim_{21}(Y',85)=12\,094\,627\,905\,536.$$
Hence, $Y'$ misses an expected hypersurface of degree 85.\\
\indent Moreover, the next values of $\adim_{21}(Y',t)$ are
\begin{gather*}
\vdim_{21}(Y',86)=-157\,230\,162\,771\,968,\\
\vdim_{21}(Y',87)=96\,757\,023\,244\,288,\\
\vdim_{21}(Y',88)=2\,366\,593\,604\,971\,209.
\end{gather*}
Hence, it seems that the virtual dimension does not provide a good prediction of the actual dimension in this case.
\end{example}

\begin{example}
We also present the analysis of the cases $k=4,5,6$ for $n=3,\ldots,50$. Using suitable procedure (in Maple or Singular) we check what is the sign of the difference $$\adim_n(X',kn+1)-\vdim_n(X',kn+1).$$
In the following cases we obtain: \\
for $k=4$
\begin{itemize}
\item unexpected hypersurfaces for $n=3,\ldots, 20$;
\item missing expected hypersurfaces for $n=21,\ldots, 50$;
\end{itemize}
for $k=5$
\begin{itemize}
	\item unexpected hypersurfaces for $n=3,\ldots, 17$, $n=42,\ldots, 50$;
	\item missing unexpected hypersurfaces for $n=18,\ldots, 41$;
\end{itemize}
for $k=6$
\begin{itemize}
	\item unexpected hypersurfaces for $n=3,\ldots, 15$, $n=37,\ldots, 50$;
	\item missing expected hypersurfaces for $n=16,\ldots, 36$.
\end{itemize}
\end{example}

\section*{Appendix}
\setcounter{section}{3}
\setcounter{theorem}{0}
\begin{lemma}
The following formula holds $$S_{n,s,t}=S_{n,s-1,t-1}+S_{n-1,s-1,t-1}.$$
\end{lemma}
\begin{proof}
Recall that $$S_{n,s,t}=\sum_{i=0}^{N(n,s)}(-1)^i\binom{s}{i}\binom{t+n-2i}{n-2i},$$ where $N(n,s)=\min\{\lfloor n/2\rfloor, s\}$.
Denote the summands in the formula as
$$a_i=(-1)^i\bbinom{s}{i}\bbinom{t+n-2i}{n-2i},\  b_i=(-1)^i\bbinom{s-1}{i}\bbinom{t+n-2i-1}{n-2i}\textnormal{ and }c_i=(-1)^i\bbinom{s-1}{i}\bbinom{t+n-2i-2}{n-2i-1}.$$
Then we want to prove that
$$S:=S_{n,s,t}-S_{n,s-1,t-1}-S_{n-1,s-1,t-1}=\sum_{i=0}^{N(n,s)}a_i-\sum_{i=0}^{N(n,s-1)}b_i-\sum_{i=0}^{N(n-1,s-1)}c_i=0.$$
Notice that for $i=0,\ldots, N(n-2,s-1)$ we have that
\begin{equation}\label{eq}
\tag{*}
a_i-b_i-c_i+r_{i-1}=r_i,
\end{equation}
where $r_{-1}=0$ and $r_i=(-1)^i\bbinom{s-1}{i}\bbinom{t+n-2i-2}{n-2i-2}$ for $i=0,\ldots, N(n-2,s-1)$.\\
Indeed, 
\begin{gather*}
a_0-b_0-c_0=\bbinom{t+n}{n}-\bbinom{t+n-1}{n}-\bbinom{t+n-2}{n-1}=\left(\bbinom{t+n-1}{n}+\bbinom{t+n-1}{n-1}\right)-\bbinom{t+n-1}{n}-\bbinom{t+n-2}{n-1}\\
=\left(\bbinom{t+n-2}{n-1}+\bbinom{t+n-2}{n-2}\right)-\bbinom{t+n-2}{n-1}=\bbinom{t+n-2}{n-2}=r_0
\end{gather*}
and
\begin{align*}
&a_i-b_i-c_i+r_{i-1}=(-1)^i\left[\bbinom{s}{i}\bbinom{t+n-2i}{n-2i}-\bbinom{s-1}{i}\bbinom{t+n-2i-1}{n-2i}-\bbinom{s-1}{i}\bbinom{t+n-2i-2}{n-2i-1}-\bbinom{s-1}{i-1}\bbinom{t+n-2i}{n-2i}\right]\\
&=(-1)^i\left[\left(\bbinom{s-1}{i}+\bbinom{s-1}{i-1}\right)\bbinom{t+n-2i}{n-2i}-\bbinom{s-1}{i}\bbinom{t+n-2i-1}{n-2i}-\bbinom{s-1}{i}\bbinom{t+n-2i-2}{n-2i-1}-\bbinom{s-1}{i-1}\bbinom{t+n-2i}{n-2i}\right]\\
&=(-1)^i\bbinom{s-1}{i}\left[\bbinom{t+n-2i}{n-2i}-\bbinom{t+n-2i-1}{n-2i}-\bbinom{t+n-2i-2}{n-2i-1}\right]\\
&=(-1)^i\bbinom{s-1}{i}\left[\left(\bbinom{t+n-2i-1}{n-2i}+\bbinom{t+n-2i-1}{n-2i-1}\right)-\bbinom{t+n-2i-1}{n-2i}-\bbinom{t+n-2i-2}{n-2i-1}\right]\\
&=(-1)^i\bbinom{s-1}{i}\left[\left(\bbinom{t+n-2i-2}{n-2i-1}+\bbinom{t+n-2i-2}{n-2i-2}\right)-\bbinom{t+n-2i-2}{n-2i-1}\right]=(-1)^i\bbinom{s-1}{i}\bbinom{t+n-2i-2}{n-2i-2}=r_i
\end{align*}
for $i=1,\ldots, N(n-1,s-1)$.\par

The equalities (\ref{eq}) allow us to simplify the formula for $S$ since we have that
\begin{equation}\label{eq2}
\tag{**}
\sum_{i=0}^{m}a_i-\sum_{i=0}^{m}b_i-\sum_{i=0}^{m}c_i=r_m
\end{equation}
for $m=0,\ldots, N(n-2,s-1)$.

Now we consider the following cases.\\
I. First, we assume that $s<\lfloor \frac{n-1}{2}\rfloor$. Then $N(n,s)=s$ and $N(n,s-1)=N(n-1,s-1)=N(n-2,s-1)=s-1$. Using (\ref{eq2}) we have that $S=0$, since
$$S=\sum_{i=0}^{s}a_i-\sum_{i=0}^{s-1}b_i-\sum_{i=0}^{s-1}c_i=a_s+r_{s-1}=(-1)^s\bbinom{t+n-2s}{n-2s}+(-1)^{s-1}\bbinom{t+n-2s}{n-2s}=0.$$
II. Now we assume that $s\geq\lfloor \frac{n-1}{2}\rfloor$ and $n=2l+1$ for some $l\in\NN$. Then $\lfloor \frac{n-1}{2}\rfloor=\lfloor \frac{n}{2}\rfloor=l$ and $\lfloor \frac{n-2}{2}\rfloor=l-1$. \par

If $s=l$, then $N(n,s)=s$ and $N(n,s-1)=N(n-1,s-1)=N(n-2,s-1)=s-1$. We get that $S=0$ by (\ref{eq2}) and the fact that $a_s+r_{s-1}=0$. \par

If $s>l$, then $N(n,s)=N(n,s-1)=N(n-1,s-1)=l$ and $N(n-2,s-1)=l-1$. We have by (\ref{eq2}) that
\begin{gather*}
S=a_l-b_l-c_l+r_{l-1}=(-1)^l\left[\bbinom{s}{l}\bbinom{t+1}{1}-\bbinom{s-1}{l}\bbinom{t}{1}-\bbinom{s-1}{l}\bbinom{t-1}{0}-\bbinom{s-1}{l-1}\bbinom{t+1}{1}\right]\\
=(-1)^l\left[\left(\bbinom{s-1}{l}+\bbinom{s-1}{l-1}\right)(t+1)-\bbinom{s-1}{l}t-\bbinom{s-1}{l}-\bbinom{s-1}{l-1}(t+1)\right]=0.
\end{gather*}
III. Finally, we take $s\geq\lfloor \frac{n-1}{2}\rfloor$ and $n=2l$ for some $l\in\NN$. Then $\lfloor \frac{n}{2}\rfloor=l$ and $\lfloor \frac{n-1}{2}\rfloor=\lfloor \frac{n-2}{2}\rfloor=l-1$.\par

If $s=l-1$ or $s=l$, then $N(n,s)=s$ and $N(n,s-1)=N(n-1,s-1)=N(n-2,s-1)=s-1$ and again, we have that $S=0$ using the equality $a_s+r_{s-1}=0$. \par

If $s>l$, then $N(n,s)=N(n,s-1)=l$ and $N(n-1,s-1)=N(n-2,s-1)=l-1$. We have that
$$S=a_l-b_l+r_{l-1}=(-1)^l\left[\bbinom{s}{l}\bbinom{t}{0}-\bbinom{s-1}{l}\bbinom{t-1}{0}-\bbinom{s-1}{l-1}\bbinom{t}{0}\right]=(-1)^l\left[\bbinom{s}{l}-\bbinom{s-1}{l}-\bbinom{s-1}{l-1}\right]=0.$$
\indent That gives all possible cases and proves the formula.
\end{proof}

\section*{Acknowledgement}
I would like to thank my advisor Marcin Dumnicki for helpful remarks and suggestions how to improve the paper.

\end{document}